\documentclass[10pt]{amsart}
\usepackage[cp1251]{inputenc}
\usepackage[english,russian]{babel}
\usepackage{amsmath}
\usepackage{longtable}
\usepackage{tikz}
\usepackage{amssymb}
\usepackage{amsfonts}

\usepackage{amscd}

\newtheorem{lemma}{Lemma}
\newtheorem{theorem}{Theorem}

\newtheorem{corollary}{Corollary}
\newtheorem{proposition}{Proposition}

\newtheorem{example}{Example}

\def\top
     {

     \begin{tabular}{c}
     \hphantom{aaaaaaaaaaaaaaaaaaaaaaaaaaaaaaaaaaaaaaaaaaaaaaaaaaaaaaaaaaaaaaaaaaaaaa} \\
     \end{tabular}

   {\flushleft
   \hspace{65mm}{\rm\small 
   } 
      }
}

\begin{document}
\renewcommand{\refname}{References}
\renewcommand{\proofname}{Proof.}
\thispagestyle{empty}

\title[Pronormal Subgroups of Odd Index]{On the pronormality of subgroups of odd index in some direct products of finite groups}
\author{{N.~V.~Maslova, D.~O.~Revin}}%

\top \vspace{1cm}
\maketitle {\small
\begin{quote}
\noindent{\sc Abstract. } A subgroup $H$ of a group $G$ is said to be {\it pronormal} in $G$ if $H$ and $H^g$ are conjugate in
$\langle H, H^g \rangle$ for each $g \in G$. Some problems in Finite Group Theory, Combinatorics, and Permutation Group Theory were solved in terms of pronormality, therefore, the question of pronormality of a given subgroup in a given group is of interest. Subgroups of odd index in finite groups satisfy a native necessary condition of pronormality.
In this paper we continue investigations on pronormality of subgroups of odd index and consider the pronormality question for subgroups of odd index in some direct products of finite groups.

In particular, in this paper we prove that the subgroups of odd index are pronormal in the direct product $G$ of finite simple symplectic groups over fields of odd characteristics if and only if the subgroups of odd index are pronormal in each direct factor of $G$.
Moreover, deciding the pronormality of a given subgroup of odd index in the direct product of simple symplectic groups over fields of odd characteristics is reducible to deciding the pronormality of some subgroup $H$ of odd index in a subgroup of $\prod_{i=1}^t \mathbb{Z}_3\wr Sym_{n_i}$, where each $Sym_{n_i}$ acts naturally on $\{1,\dots, n_i\}$, such that $H$ projects onto $\prod_{i=1}^t Sym_{n_i}$. Thus, in this paper we obtain a criterion of pronormality of a subgroup $H$ of odd index in a subgroup of $\prod_{i=1}^t \mathbb{Z}_{p_i}\wr Sym_{n_i}$, where each $p_i$  is a prime and each $Sym_{n_i}$ acts naturally on $\{1,\dots, n_i\}$, such that $H$ projects onto $\prod_{i=1}^t Sym_{n_i}$.
\medskip

 \end{quote}
}


\section{Introduction}\label{Intr}

Throughout the paper we consider only finite groups, and henceforth the term group means finite group.
Our further terminology and notation are mostly standard and can be found in \cite{Kleidman_Liebeck}.
However, we will denote by $Sym_n$ and $Alt_n$ the symmetric group and the alternating group of degree $n$, respectively.

A subgroup $H$ of a group $G$ is said to be {\it pronormal} in $G$ (notation $H \,prn\, G$) if $H$ and $H^g$ are conjugate in $\langle H, H^g \rangle$ for each $g \in G$. Some of well-known examples of pronormal subgroups are the following: normal subgroups; maximal subgroups; Sylow subgroups; Sylow subgroups of proper normal subgroups; Hall subgroups of solvable groups; Hall subgroups of proper normal subgroups of solvable groups. The following assertion by Ph.~Hall demonstrates a close connection between properties of permutation representations of finite groups and pronormality of their subgroups.

\begin{proposition}[{\rm \cite[Theorem~6.6]{Hall}}]\label{HallProp} Let $G$ be a group and $H\le G$. Then $H$ is pronormal in $G$ if and only if in any transitive permutation representation of $G$, the subgroup $N_G(H)$ acts transitively on the set $fix(H)$ of fixed points of $H$.
\end{proposition}

Some problems in Finite Group Theory as well as in Combinatorics and in Permutation Group Theory are solved in terms of pronormality. For example, consider the well-known Frattini Argument: {\it if $G$ is a finite group with normal subgroup $H$, and if $P$ is a Sylow subgroup of $H$, then $G=N_{G}(P)H$}. It is easy to see that the condition ''$P$ is a Sylow subgroup of $H$'' can be replaced to the following more general condition ''$P$ is a pronormal subgroup of $G$'', and the implication remains true. Furthermore, the concept of pronormal subgroup in some sense is universal with respect to the Frattini Argument: {in  the introduced notation, a subgroup $P$ is pronormal in $G$ if and only if $P$ is pronormal in $H$ and $G=N_{G}(P)H$ (see \cite[Lemma~4]{Guo_Mas_Rev}). In 1971, T.~Peng~\cite{Peng} showed that if $G$ is solvable, then a subgroup $P$ is pronormal in $G$ if and only if $P$ possesses the {\it Frattini property} in $G$: $L=N_{L}(P)H$ for all subgroups $L,H \le G$ such that $H \unlhd L$ and $P \le H$. In particular, solvable groups have Frattini factorizations with respect to Hall subgroups of their normal subgroups. Recently E.~P.~Vdovin and the second author~\cite{Vdovin_Revin_2015} have showed that the existence of a $\pi$-Hall subgroup in a group $G$ for some set $\pi$ of primes is equivalent to the existence of a pronormal $\pi$-Hall subgroup in each normal subgroup of $G$. Thus, the $E_\pi$-groups possess corresponding Frattini factorizations.

Moreover, the concept of pronormal subgroup is closely connected to the Cayley isomorphism problem as follows. According to L.~Babai~\cite{Babai}, a finite group $G$ is a {\it CI-group} (abbreviation of Cayley isomorphism property) if between any two isomorphic relational structures on the group $G$ as underlying set which are invariant under the group $$G_R = \{gR \mid  g \in G\}$$
of right multiplications $gR : x \rightarrow xg$ (where $g, x \in G$), there exists an isomorphism
which is at the same time an automorphism of $G$.  L.~Babai~\cite{Babai} proved that a group $G$ is a CI-group if and only if the subgroup $G_R$ is pronormal in $Sym(G)$, in particular, if $G$ is a CI-group, then $G$ is abelian; with using this Babai’s result, P. Palfy~\cite{Palfy} has obtained the complete classification of CI-groups. Moreover, in~\cite{Babai}, L.~Babai characterized combinatorial CI-objects (in particular, Cayley graphs of groups) in terms of concepts close to pronormality. These characterizations are useful tools, for example, in researches of Cayley isomorphism problem concerning undirected graphs (see, for example, \cite{Dobson}).


Thus, the following problem is of interest.

\medskip

\noindent{\bf General Problem.} Is a given subgroup $H$ pronormal in a given group~$G$?

\medskip

Ch.E.~Praeger \cite{Praeger} investigated pronormal subgroups of permutation groups. She has obtained the following result.

\begin{proposition}\label{PrThm}  Let $G$ be a transitive permutation group on a set $\Omega$ of $n$ points, and let $K$ be a non-trivial pronormal subgroup of $G$. Suppose that $K$ fixes exactly $f$ points of $\Omega$. Then $f \le \frac{1}{2}(n-1)$, and if $f = \frac{1}{2}(n-1)$, then $K$ is transitive on its support in $\Omega$, and either $G \ge Alt(n)$, or
$G = GL_d(2)$ acting on the $n = 2^d - 1$ non-zero vectors, and $K$ is the pointwise stabilizer of a hyperplane.
\end{proposition}

Thus, if in some transitive permutation representation of $G$, $|fix(H)|$ is too big, then $H$ is not pronormal in $G$. Therefore, first of all, it is important to consider General Problem for a subgroup $H$ of a group $G$ such that $H$ contains a subgroup $S$ which is pronormal in $G$ since in this case $H$ is already satisfying a necessary condition of pronormality in $G$. Thus, it is interesting solve General Problem for  overgroups of Sylow subgroups in finite groups, in particular, for subgroups of odd index in a finite group $G$ which are exactly overgroups of Sylow $2$-subgroups of $G$.

\medskip

In 2012, E.~P.~Vdovin and the second author \cite{Vdovin_Revin} proved that the Hall subgroups are pronormal in all simple groups and,
guided by the analysis in their proof, they conjectured that any subgroup of odd index of a simple group is pronormal in this group.
This conjecture was disproved in \cite{Kond_Mas_Rev2,Kond_Mas_Rev3}. The following problem naturally arose.

\medskip
\noindent{\bf Problem A.} Determine finite simple groups in which the subgroups of odd index are pronormal.
\medskip

Problem~A was investigated in \cite{Kond_Mas_Rev1,Kond_Mas_Rev2,Kond_Mas_Rev3,Kond_Mas_Rev4,Kond_Mas_Rev5}, and in this moment, Problem A is still open only for some linear and unitary simple groups over fields of odd characteristics. More detailed surveys of investigations on pronormality of subgroups of odd index in finite (not necessary simple) groups can be found in survey papers~\cite{GR_Surv,KMR_Surv}. These surveys contain new results and some conjectures and open problems. In particular, in~\cite[Section~10]{KMR_Surv} we have provided ideas how to reduce solving General Problem for a subgroup of odd index in a non-simple group to solving General Problem for some subgroup of odd index in a  group of smaller order. In connection with this, the following problem arose.

\medskip
\noindent{\bf Problem B.} Determine direct products of finite simple groups in which the subgroups of odd index are pronormal.
\medskip

A detailed motivation for Problem~B was provided in \cite{Guo_Mas_Rev} and in the survey paper \cite{KMR_Surv}. In general, the question of pronormality of a subgroup in a direct product of finite groups is natural and was studied in some special cases. For example, B. Brewster, A. Mart\'{\i}nez-Pastor, and M.~D. P\'{e}rez-Ramos \cite{BMPPR} have given criteria to characterize abnormal, pronormal and locally pronormal subgroups of a direct product of two finite groups $A \times B$, under hypotheses of
solvability for at least one of the factors, either $A$ or $B$ (see for some details Lemma~\ref{ProInDirProd} in Section~\ref{Prelim}).

\medskip

Note that the subgroups of odd index are pronormal in groups with self-normalizing Sylow $2$-subgroups (see
\cite[Lemma~5]{Kond_Mas_Rev1}), and Sylow $2$-subgroups are self-normalizing in the direct product of groups with self-normalizing
Sylow $2$-subgroups (see, for example, \cite[Lemma~2]{Kond_Mas_Rev3}). Taking into account that the Sylow $2$-subgroups are
self-normalizing in many nonabelian simple groups \cite{Kondrat'ev}, we conclude that the situation when the subgroups of odd index
are pronormal in a direct product of finite simple groups occurs rather often. However, there are examples of nonabelian simple groups
$G$ (in which the Sylow $2$-subgroups are not self-normalizing) such that all the subgroups of odd index are pronormal in $G$, but the
group $G \times G$ contains a non-pronormal subgroups of odd index (see \cite[Proposition~1]{Guo_Mas_Rev}).

In this paper we consider direct products of finite simple symplectic groups. If $q\equiv \pm 3\pmod{8}$, then the Sylow
$2$-subgroups are not self-normalizing in the group $PSp_{2n}(q)$ for any $n$ (see \cite{Kondrat'ev}). However, we prove the following
theorem.

\begin{theorem}\label{ProdPSp2w2w(22k+1)} Let $G =\prod\limits_{i=1}^t G_i$, where for each $i \in \{1, . . . , t\}$, $G_i \cong
Sp_{n_i}(q_i)$ for odd~$q_i$. Then the following statements are equivalent{\rm:}

$(1)$ all the subgroups of odd index are pronormal in $G${\rm;}

$(2)$ for each $i\in \{1,\ldots,t\}$, all the subgroups of odd index are pronormal in $G_i${\rm;}

$(3)$ for each $i \in \{1,\ldots,t\}$, if $q_i \equiv \pm 3 \pmod{8}$, then $n_i$ is either a power of $2$ or is a number of the form
$2^w(2^{2k}+1)$ for non-negative integers $w$ and $k$.

\end{theorem}

Note that if $q$ is odd, then $|Z(Sp_{2n}(q))|=2$ for any $n$. Therefore, if $H$ is a subgroup of odd index in the group
$G=\prod\limits_{i=1}^t Sp_{2n_i}(q_i)$, where all $q_i$ are odd, then ${Z(G)\le H}$. Thus, applying Lemma~\ref{Quot} from
Section~\ref{Prelim} below, we obtain the following corollary.

\begin{corollary}\label{ProdPSp2w2w(22k+1)} Let $G =\prod\limits_{i=1}^t G_i$, where for each $i \in \{1, . . . , t\}$, $G_i \cong
PSp_{n_i}(q_i)$ for odd~$q_i$. Then the following statements are equivalent{\rm:}

$(1)$ all the subgroups of odd index are pronormal in $G${\rm;}

$(2)$ for each $i\in \{1,\ldots,t\}$, all the subgroups of odd index are pronormal in $G_i${\rm;}

$(3)$ for each $i \in \{1,\ldots,t\}$, if $q_i \equiv \pm 3 \pmod{8}$, then $n_i$ is either a power of $2$ or is a number of the form
$2^w(2^{2k}+1)$ for non-negative integers $w$ and $k$.

\end{corollary}

Moreover,  guided by the analysis in the proof of Theorem~\ref{ProdPSp2w2w(22k+1)} (see Remark~1 in the end of Section~\ref{ProofT1}), we conclude that solving General Problem for a given subgroup $H$ of odd index in the direct product
of symplectic groups over fields of odd characteristics is reducible to solving General Problem for some subgroup $H^*$ (depending on $H$) of odd index in some group $$K\le \prod\limits_{i=1}^{t} \mathbb{Z}_3\wr Sym_{m_i}$$ such
that $H \le K$ and $H$ projects onto each $Sym_{m_i}$. In this paper, we obtain a criterion of pronormality of such a subgroup
$H$ in such a group $K$, see { Theorem~\ref{SOIinProdC3wrSn}} formulated in Section~\ref{SPT2}.

\smallskip

Thus, General Problem for a subgroup of odd index in a direct product of simple symplectic groups over fields of odd
characteristics (in particular, in a simple symplectic group over a field of odd characteristic) can be formally solved with using
inductive reasonings. As an example, the detailed solution of General Problem for an arbitrary subgroup of odd index in the group
$PSp_6(3)$ is presented in Example~\ref{ExPSp6(3)}, see Section~\ref{SPT2} after Theorem~\ref{SOIinProdC3wrSn}.

\smallskip

The following problems naturally arise.

\medskip
\noindent{\bf Problem 1.} Find a criterion of pronormality of a subgroup of odd index in the direct product of simple symplectic groups over
fields of odd characteristics (in particular, in a simple symplectic group over a field of odd characteristic).

\medskip
\noindent{\bf Problem 2.} Provide an effective algorithm which solves General Problem for a subgroup of odd index in the direct
product of simple symplectic groups over fields of odd characteristics (in particular, in a simple symplectic group over a field of
odd characteristic).
\medskip

\section{ Preliminaries and auxiliary results}\label{Prelim}

The largest integer power of a prime $p$ dividing a positive integer $k$ is called the {\it $p$-part} of $k$ and is denoted by $k_p$.
\smallskip

Let $m$ and $n$ be non-negative integers with the binary expansions $$m=\sum_{i=0}^\infty a_i\cdot 2^i \mbox{ and }
n=\sum_{i=0}^\infty b_i\cdot 2^i,$$ where $a_i,b_i\in\{0,1\}$ for every $i$. We write $m\preceq n$ if $a_i\leq b_i$ for each $i$ and
$m\prec n$ if additionally $m \not = n$.
It is clear that $m\preceq n$ if and only if $n-m\preceq n$.

\smallskip

Let $G$ be a group and $p$ be a prime. We write $H \le_p G$ if $H \le G$ and $|G:H|_p=1$.

\smallskip

Let $\mathbb{X}_p$ be the class of groups with self-normalizing Sylow $p$-subgroups.

\smallskip

Let $\mathbb{Y}_p$  be the class of groups in which the overgroups of Sylow $p$-subgroups are pronormal. Lemma~\ref{Overgroup} below
shows that $\mathbb{X}_p \subset \mathbb{Y}_p$. Note that $\mathbb{Y}_2$ is exactly the class of all groups in which the subgroups of
odd index are pronormal.

\medskip

The following two lemmas deal with overgroups of Sylow subgroups in direct products of finite groups and are of independent interest.

\begin{lemma}\label{Projections} Let $p$ be a prime and $Q$ be a subgroup of the group ${L=L_1\times L_2 \times \ldots \times L_n}$, and let $\pi_i: L \rightarrow L_i$ be the projection for each $i \in \{1,\ldots, n\}$. Assume that $Q\le_p L$ and for some~$i$,  $\pi_i(Q)=L_i$ and $L_i/O_p(L_i)$ is an almost simple group such that $(L_i/O_p(L_i))/Soc(L_i/O_p(L_i))$ is a $p$-group. Then $L_i \le Q$.
\end{lemma}
\begin{proof} Note that $O_p(L_i) \le O_p(L)$. Therefore, $O_p(L_i)\le Q$ since $Q \le_p L$. Thus, we can assume assume that
$O_p(L_i)=1$, and so $L_i$ is almost simple.

Now one can prove this assertion repeating all the reasonings from the proof of {\cite[Lemma~9]{Guo_Mas_Rev}} and replacing the prime $2$ into an arbitrary prime $p$ as follows. Since $L_i \unlhd L$, we have $Q \cap L_i \unlhd Q$, and therefore, $\pi_i(Q \cap L_i)$ is a normal subgroup of $\pi_i(Q) = L_i$.
Choose $S \in Syl_p(L)$ such that $S \le Q$. Then $S \cap L_i \in Syl_p(L_i)$ and $$S \cap L_i = \pi_i(S \cap L_i) \le \pi_i(Q \cap L_i).$$
Therefore, $\pi_i(Q \cap L_i) \unlhd L_i$ and $\pi_i(Q \cap L_i)\le_p L_i$. The group $L_i$ is almost simple and $L_i/ Soc(L_i)$ is a $p$-group, so $\pi_i(Q \cap L_i) = L_i$, whence $L_i \le Q$.
\end{proof}

\begin{lemma}\label{XtimesYprojections} Let  $G=X\times Y$ for groups $X$ and $Y$ and let
$$\xi:G\rightarrow X\quad\text{ and }\quad \eta:G\rightarrow Y$$
be the projections. Assume that
$X\in \mathbb{X}_p$ and $H\le_p G$. Then $\xi(H)=H\cap X$ and $\eta(H)=H\cap Y$. In particular,
$$H=\langle\xi(H),\eta(H)\rangle=\xi(H)\times \eta(H).$$
\end{lemma}

\begin{proof}
Since $H\cap Y\le\eta(H)$ and $h=\xi(h)\eta(h)$ for every $h\in H$, it is sufficient to establish that $H\cap X=\xi(H)$. Now $X\unlhd
G$ implies that $H\cap X\unlhd H$ and $$H\cap X=\xi(H\cap X)\unlhd\xi(H).$$ On the other hand, $H\le_p G$ means $H\cap X\le_p X$. If
$P\in Syl_p(H\cap X)$, then $P\in Syl_p(X)$ and $N_X(P)=P$ since $X\in \mathbb{X}_p$. The Frattini Argument implies that
$$\xi(H)=(H\cap X)N_{\xi(H)}(P)\le (H\cap X)N_{X}(P)=(H\cap X)P=H\cap X.$$
\end{proof}

\medskip

In the following series of lemmas we provide some general properties of pronormal subgroups in finite groups.

\begin{lemma}[\rm See {\cite[Lemma~3]{Kond_Mas_Rev1}} and {\cite[Ch.~1, Prop.~6.4]{DH}}]\label{Quot} Let   $A\unlhd G$ and $H\leq G$.
Then the following statements hold{\rm:}

$(1)$ if $H \,prn\, G$, then $HA/A\,prn\, G/A${\rm;}

$(2)$ $H\,prn\, G$ if and only if $HA/A\,prn\, G/A$ and $H\,prn\, N_G(HA)${\rm;}

$(3)$ if $A \le H$, then $H \,prn\, G$ if and only if $H/A\,prn\, G/A${\rm;}

$(4)$ if $p$ is a prime and $H \le_p G$, then $H\,prn\, G$ if and only if $H/O_p(G)\,prn\, G/O_p(G)$.
\end{lemma}

The following assertion is a direct corollary of Proposition~\ref{HallProp}, however, the assertion was independently proved in~\cite[Lemma~5]{Vdovin_Revin}.

\begin{lemma}\label{NormSyl}
Suppose that $G$ is a group and $H \le G$. Assume also that $H$ contains a Sylow subgroup $S$ of $G$. Then $H$ is pronormal in $G$ is
and only if $H$ and $H^g$ are conjugate in $\langle H, H^g \rangle$ for each $g \in N_G(S)$.
\end{lemma}

\begin{lemma}[\rm See {\cite[Lemma~5]{Kond_Mas_Rev1}}]\label{Overgroup} Suppose that $H$ and $M$ are subgroups of a group  $G$ and $H
\le M$.

$(1)$ If $H\,prn\, G$, then $H\,prn\, M${\rm;}

$(2)$ If $S \le H$ for some Sylow subgroup $S$ of $G$, $N_G(S) \le M$, and $H\,prn\, M$, then $H\,prn\, G$.

\end{lemma}

\begin{lemma}[\rm See {\cite[Theorem~1]{Kond_Mas_Rev2}}]\label{CritComplAbel} Let $H$ and $V$ be subgroups of a group $G$ such that
$V$ is an abelian normal subgroup of $G$ and $G=HV$. Then $H$ is pronormal in $G$ if and only if $U=N_U(H)[H,U]$ for any $H$-invariant
subgroup $U$ of $V$.

\end{lemma}

\begin{lemma}[\rm See {\cite[Lemma~11]{Guo_Mas_Rev}}]\label{PronHomImag} Let $G=V \rtimes B$, where $V$ is an abelian normal subgroup
of $G$ and $B \le G$, and let $H \le G$. Define $\overline{\phantom{x}}: G \rightarrow B$ such that $\overline{g}=b$, where $g=vb$ for
$v\in V$ and $b \in B$. Then $\overline{H} \,prn\, \overline{H}V$ implies $H \,prn\, HV$.

\end{lemma}

\begin{lemma}[\rm See {\cite[Lemma~6]{Guo_Mas_Rev}}]\label{SelfNormSylInQuot} Let $N \trianglelefteq G$, $G/N \in \mathbb{X}_p$, and
$H\le_p G$. Then $H\,prn\, G$ if and only if $H\,prn\, HN$.\end{lemma}

\begin{lemma} {\rm \cite[Propositions 2.1, 4.3, 4.4 and Corollary 4.7]{BMPPR}} \label{ProInDirProd} Let $G=G_1\times G_2$  and $H\le
G$. For $i\in \{1,2\}$, denote by  $\pi_i$ the projection $G\rightarrow G_i$ and set
$$C_i=\{x\in G_i\mid [x, \pi_i(H)]\le G_i\cap H\}.$$ Then the following statements hold.
\begin{itemize}
\item[$(1)$] $C_i=N_{G}(H)\cap G_i$.
\item[$(2)$] If $\pi_i(H)  \,prn\, G_i$ for every $i\in \{1,2\}$ and $$N_G(H)=\langle N_{G_i}(\pi_i(H)) \mid
    i \in \{1,2\}\rangle=N_{G_1}(\pi_1(H))\times N_{G_2}(\pi_2(H)),$$ then $H  \,prn\, G$.
\item[$(3)$] If  one of $G_i$, where $i \in \{1,2\}$, is solvable, then  $H  \,prn\, G$ if and only if $\pi_i(H)  \,prn\, G_i$ for every $i\in \{1,2\}$
    and $$N_G(H)=\langle N_{G_i}(\pi_i(H)) \mid i \in \{1,2\}\rangle=N_{G_1}(\pi_1(H))\times N_{G_2}(\pi_2(H)).$$
\item[$(4)$]  If  one of  $G_i$, where $i \in \{1,2\}$, is solvable, then  $H  \,prn\, G$ if and only if $\pi_i(H)  \,prn\, G_i$  and $N_{G_i}(\pi_i(H))\le C_i$ for each $i \in \{1,2\}$.
\end{itemize}
\end{lemma}

\medskip

 Two further assertions give criteria of pronormality of overgroups of Sylow subgroups in extensions of finite groups with special properties.

\begin{lemma}\label{XtimesY} Assume that $p$ is a prime and $X$ and $Y$ are groups such that $X\in \mathbb{X}_p$.
Let $H \le_p X \times Y$ and let
$$\xi:G\rightarrow X\quad\text{ and }\quad \eta:G\rightarrow Y$$
be the projections.
Then the following statements hold{\rm:}

$(i)$ $H$ is pronormal in $X \times Y$ if and only if $\eta(H)$ is pronormal in $Y$.

$(ii)$ $X\times Y\in \mathbb{Y}_p$ if and only if $Y\in\mathbb{Y}_p$.
\end{lemma}

\begin{proof} Prove Statement $(i)$. By Lemma~\ref{Quot}, if $H$ is pronormal in $X \times Y$, then $\xi(H)$ is pronormal in $\xi(X\times Y)=X$ and $\eta(H)$ is pronormal in $\eta(X \times Y)=Y$.

Show the converse. Suppose that  $\eta(H)$ is pronormal in $Y$ and note that $\xi(H)$ is pronormal in $X$ because $X \in \mathbb{X}_p\subset \mathbb{Y}_p$ and $\xi(H)\le_p \xi(X\times Y)=X$. Since $H \le_p X\times Y$ and $X \in \mathbb{X}_p$, Lemma~\ref{XtimesYprojections} implies that $H=\xi(H)\times \eta(H)$. Therefore, $$N_{X\times Y}(H)=N_X(\xi(H))\times N_Y(\eta(H)).$$ Now $H$ is pronormal in $X\times Y$ by Lemma~\ref{ProInDirProd} part~$(2)$.

\medskip

Statement~$(ii)$ follows directly from Statement~$(i)$.





\end{proof}

\begin{proposition}\label{GeneralCriterium} Let $G$ be a group, $A\unlhd G$, where $G/A \in \mathbb{X}_p$ and $A \in \mathbb{Y}_p$.
Let $T$ be a Sylow $p$-subgroup of $A$. Then

$(i)$ if $H \ge S$ for some $S\in Syl_p(G)$, $T = A\cap S$, $Y=N_{A}(H\cap A)$ and $Z=N_{H\cap A}(T)$,
then $Z$ is normal in both $N_H(T)$ and $N_Y(T)$ and the following conditions are equivalent{\rm:}

$\mbox{ }$ $\mbox{ }$ $\mbox{ }$  $(1)$ $H$ is pronormal in $G${\rm;}

$\mbox{ }$ $\mbox{ }$ $\mbox{ }$ $(2)$ $N_H(T)/Z$  is pronormal in $(N_H(T) N_Y(T) )/Z${\rm;}

$\mbox{ }$ $\mbox{ }$ $\mbox{ }$ $(3)$ $N_H(T)$ is pronormal in $N_H(T) N_Y(T)$.

$(ii)$ $G \in \mathbb{Y}_p$ if and only if $N_G(T)/T \in \mathbb{Y}_p$.

\end{proposition}

\begin{proof} Statement $(ii)$ is  {\cite[Theorem~1]{Guo_Mas_Rev}}.

\medskip

To prove Statement $(i)$ we need to go throw the proof of {\cite[Lemma~15]{Guo_Mas_Rev}} and generalize the reasonings.
Define $$X=N_{HA}(H\cap A).$$

Note that $H\le X\text{, therefore, }$ $$ X=N_{HA}(H\cap A)=HN_A(H\cap A)=HY.$$
By the Frattini Argument,
$$H=N_H(T)(H\cap A) \mbox{ and }  Y=N_Y(T)(H\cap A).$$
Moreover, $N_H(T)$ normalizes $N_Y(T)$ and
$$N_H(T) \cap (H \cap A)=N_Y(T) \cap (H\cap A)=N_{H\cap A}(T)=N_G(T)\cap (H\cap A) =(N_H(T) N_Y(T))\cap (H\cap A).$$
Note that by Lemma~\ref{Quot}, $$H = N_H(T)(H\cap A)  \,prn\, X = N_{HA}(H\cap A)=(H\cap A)N_Y(T)N_H(T)$$ if and only if
$$N_H(T)(H\cap A)/(H\cap A) \,prn\, (N_Y(T)N_H(T)(H\cap A))/ (H\cap A)$$
if and only if $$N_H(T)/N_{H\cap A}(T) \,prn\, N_H(T) N_Y(T) / N_{H\cap A}(T)$$ if and only if $$N_H(T) \,prn\, N_H(T) N_Y(T).$$

\medskip

Thus, we have proved that the following conditions are equivalent{\rm:}

$(1)$ $H$ is pronormal in $N_{HA}(H\cap A)${\rm;}

$(2)$ $N_H(T)/Z$  is pronormal in $(N_H(T) N_Y(T) )/Z${\rm;}

$(3)$ $N_H(T)$ is pronormal in $N_H(T) N_Y(T)$.

\medskip

Now the condition $A \in \mathbb{Y}_p$ and {\cite[Lemma~15, part~$(2)$]{Guo_Mas_Rev}} imply that $H$ is pronormal in $N_{HA}(H\cap A)$
if and only if $H$ is pronormal in $HA$.

\smallskip

The condition $G/A \in \mathbb{X}_p$ and Lemma~\ref{SelfNormSylInQuot} imply that $H$ is pronormal in $G$ if and only if $H$ is
pronormal in $HA$.

\end{proof}


\medskip

In the further series of lemmas we provide some properties of subgroups of odd index in finite groups of special type.


\begin{lemma}[{\rm See {\cite{Kondrat'ev}}}]\label{N_G(S)=S} Let $G=PSp_n(q)$, where $q$ is odd, and $S \in Syl_2(G)$.

$(1)$ If $q \equiv \pm 1 \pmod 8$, then $N_G(S)=S$.

$(2)$ If $q \equiv \pm 3 \pmod 8$ and $n=2^{s_1}+\dots+2^{s_t}$ for $s_1>\dots>s_t\geq 0$, then $N_G(S)/S$ is elementary abelian of
order $3^t$.

\end{lemma}

\begin{lemma}[{\rm See \cite{Maslova1} and \cite{Maslova10}}] \label{CMSOIPSp} Let $G=Sp_{2n}(q)$, where $n \ge 1$ and $q$ is
odd{\rm;} let $V$ be the natural module of $G$. A subgroup $H$ is a maximal subgroup of odd index in $G$ if and only if one of the
following statements holds{\rm:}

$(1)$ $H \cong Sp_{2n}(q_0)$, where $q=q_0^r$ and $r$ is an odd prime, is the centralizer in $G$ of a field automorphism of order
$r${\rm;}

$(2)$ $H \cong Sp_{2m}(q)\times Sp_{2(n-m)}(q)$ is the stabilizer of a non-degenerate subspace of dimension $2m$ of $V$, and  $n \succ
m${\rm;}

$(3)$ $H \cong Sp_{2m}(q)\wr Sym_t$ is the stabilizer of an orthogonal decomposition $V =\bigoplus V_i$ into a sum of pairwise
isometric non-degenerate subspaces $V_i$ of dimension $2m$, $m = 2^w$, $w$ is a non-negative integer, and $n=mt${\rm;}

$(4)$ $n=1$ and $H \cong  SL_2(q_0).2$ is the centralizer in $G$ of a field automorphism of order~$2${\rm;}

$(5)$ $n=1$, $H/Z(G) \cong Alt_4$,  $q$ is prime, and $q=5$ or $q \equiv \pm 3, \pm 13\pmod {40}${\rm;}

$(6)$ $n=1$, $H/Z(G) \cong Sym_4$,  $q$ is prime, and $q \equiv \pm 7 \pmod {16}${\rm;}

$(7)$ $n=1$, $H/Z(G) \cong Alt_5$,  $q$ is prime, and $q \equiv 11, 19, 21, 29 \pmod {40}${\rm;}

$(8)$ $n=1$, $H/Z(G) \cong D_{q+1}$,  and  $7 < q \equiv 3 \pmod 4${\rm;}

$(9)$ $n=2$, $H/Z(G) \cong 2^4.Alt_5$,  $q$ is prime, and $q \equiv \pm 3 \pmod 8$.

\end{lemma}







\begin{lemma}\label{NormInWrProdPSpSn} Let $L=Sp_{2^w}(q)$ for odd $q$ and $w\ge 1$, $P \in Syl_2(L)$, and $G={L\wr Sym_n}$, where
$Sym_{n}$ acts naturally on $\{1,\dots, n\}$. Let $K=\prod\limits_{i=1}^n K_i$, where each $K_i$ is isomorphic to $L$, be a normal
subgroup of $G$, which coincides with the base of this wreath product. Let $S\in Syl_2(G)$, $T=K\cap S$, and $T_i =K_i\cap T$. Then
the following statements hold{\rm:}

$(1)$ $T_i \cong P$ for each $i${\rm;}

$(2)$ $T = \prod\limits_{i=1}^n T_i \cong \underbrace{P\times\dots\times P}\limits_{n\text{ \rm times}}${\rm;}

$(3)$ $N_K(T) = \prod\limits_{i=1}^n N_{K_i}(T_i) \cong \underbrace{N_L(P)\times\dots\times N_L(P)}\limits_{n\text{ \rm
times}}${\rm;}

$(4)$ $N_G(T)\cong N_L(P)\wr Sym_n$, where $Sym_{n}$ acts naturally on $\{1,\dots, n\}${\rm;}

$(5)$ $N_G(T)/T\cong (N_L(P)/P)\wr Sym_n \cong Z\wr Sym_n$, where $Sym_{n}$ acts naturally on $\{1,\dots, n\}$, $Z \cong \mathbb{Z}_3$
if ${q\equiv \pm 3\pmod{8}}$, and $Z$ is trivial if ${q\equiv \pm 1\pmod{8}}$.

\end{lemma}

\begin{proof} Assertions $(1)$–$(3)$ are obvious. Assertion $(4)$ easily follows from $(3)$ and the Frattini Argument. Assertion $(5)$
follows from $(4)$ and Lemma~\ref{N_G(S)=S}.

\end{proof}

\medskip

In the following series of lemmas we provide some results on pronormality of subgroups of odd index in finite groups.

\begin{lemma}[\rm See {\cite[Corollary]{Kond_Mas_Rev2}}]\label{ComplAwrSn} Let $G=A\wr Sym_n= HV$ be the wreath product of an abelian
group $A$ by the symmetric group $H=Sym_n$ acting naturally on the set $\{1,\ldots,n\}$, where $V$ denotes the base of the wreath
product. Then the subgroup $H$ is pronormal in $G$ if and only if $(|A|,n)=1$.
\end{lemma}

\begin{lemma}[\rm See {\cite[Theorem~2]{Guo_Mas_Rev}}]\label{ProdAbWRSn} Let $A$ be an abelian group and $$G=\prod\limits_{i=1}^t
(A\wr Sym_{n_i}),$$ where each $Sym_{n_i}$ acts naturally on $\{1,\dots, n_i\}$. Then the subgroups of odd index are pronormal in $G$
if and only if for any positive integer $m$, the inequality $m \preceq  n_i$ for some $i$ implies that $(|A|,m)$ is a power of $2$.
\end{lemma}

\begin{lemma}[\rm See {\cite{Kond_Mas_Rev4}}]\label{PSpPRN} Let $G=PSp_{2n}(q)$. Then $G \in \mathbb{Y}_2$ if and only if one of the
following statements holds:

$(1)$ $q \not \equiv \pm 3 \pmod 8${\rm;}

$(2)$ $n$ is of the form $2^w$ or $2^w(2^{2k}+1)$, where $k$ and $w$ are non-negative integers.

\end{lemma}

\begin{lemma}[\rm See {\cite[Theorem~3]{Guo_Mas_Rev}}]\label{ProdPSp2w} Let $G =\prod\limits_{i=1}^t G_i$, where for each $i \in \{1,
\dots , t\}$, $G_i \cong Sp_{2n_i}(q_i)$, each $q_i$ is odd, and each $n_i$ is a power of $2$. Then all the subgroups of odd index are
pronormal in $G$.
\end{lemma}

\section{Proof of Theorem~\ref{ProdPSp2w2w(22k+1)}}\label{ProofT1}

$(3) \Rightarrow (1)$ Let us tell that a group $G$ satisfy {\it condition $(*)$} if the following statements hold:

\begin{itemize}
\item[---] $G =\prod\limits_{i=1}^t G_i$, where $G_i \cong Sp_{2n_i}(q_i)$ for each $i \in \{1, . . . , t\}${\rm;}

\item[---] each $q_i$ is odd{\rm;}

\item[---] if $q_i \equiv \pm 3 \pmod{8}$ for some $i$, then $n_i$ is either a power of $2$ or is a number of the form
    $2^{w_i}(2^{2k_i}+1)$, where $w_i$ and $k_i$ are non-negative integers.

\end{itemize}

 Assume that $G$ is a group of the smallest order satisfying~$(*)$, such that $G$ contains a non-pronormal subgroup $H\le_2
 G$, and take some $S \in Syl_2(G)$ with $S \le H$. By Lemma~\ref{N_G(S)=S}, $Sp_{2n_i}(q)\in\mathbb{X}_2$ if $q_i\equiv\pm1\pmod 8$, therefore, Lemma~\ref{XtimesY} and the minimality of $G$ implies that $q_i\equiv \pm 3\pmod 8$ for each $i \in \{1, \ldots t\}$.


Let $\pi_i: G \rightarrow G_i$ be the projection for each $i \in \{1, \ldots t\}$. If $\pi_i(H)=G_i$ for some~$i$, then $G_i \le H$ by
Lemma~\ref{Projections}. Thus, $H/G_i$ is a non-pronormal subgroup of odd index in $G/G_i \cong \prod\limits_{j \not =i} G_i$ by
Lemma~\ref{Quot}. But the group $G/G_i$ satisfy condition~$(*)$, a contradiction to the minimality of $G$.

So, for each $i$, there exists a maximal subgroup $M_i < G_i$ such that $\pi_i(H) \le M_i$. Thus, for any $i$,
$$H \le M(i)=\prod\limits_{j \not = i} G_j \times M_i.$$
Possibilities for $M_i$ are listed in Lemma~\ref{CMSOIPSp}.
\smallskip

Assume that for some $i$, $M_i \cong Sp_{2n_i}(\tilde{q}_{i})$, where $q_i=\tilde{q}_{i}^{r_i}$ and $r_i$ is an odd prime. Note that
$q_i \equiv \pm 3 \pmod 8$ implies $\tilde{q}_{i} \equiv \pm 3 \pmod 8$. It is easy to see that in this case $M(i)$ satisfies
condition $(*)$, and $H$ is pronormal in $M(i)$ by the minimality of~$G$.  Moreover, $N_G(S)=N_{M(i)}(S)$ by Lemma~\ref{N_G(S)=S}.
Thus, $H \,prn\, G$ by Lemma~\ref{Overgroup}, a contradiction.

\smallskip

Assume that for some $i$, $M_i$ is the stabilizer of a non-degenerate subspace of dimension $2m_i$ of the natural module of $G_i$, and
$n_i \succ  m_i$. Note that $$M_i \cong Sp_{2m_i}(q_i) \times Sp_{2(n_i-m_i)}(q_i).$$ In this case $M(i)$ satisfies condition $(*)$ since $n_i \succ  m_i$.
By the minimality of $G$, $H \,prn\, M$.
Moreover, since $n_i \succeq  m_i$, we have $N_G(S)=N_{M(i)}(S)$ by Lemma~\ref{N_G(S)=S}. Thus, $H \,prn\, G$ by
Lemma~\ref{Overgroup}, a contradiction.

\smallskip

So, by Lemma~\ref{CMSOIPSp}, for each $i$ such that $n_i$ is not a power of $2$, $M_i$ is the stabilizer of an orthogonal
decomposition of the corresponding natural module $V_i$ of $G_i$ into a sum of pairwise isometric non-degenerate subspaces of
dimension $2s_i$, and $s_i = 2^{w_i}$ for a non-negative integer $w_i$. Note that by~\cite[Proposition~4.9.2]{Kleidman_Liebeck}, in
this case we have $$M_i \cong Sp_{2s_i}(q_i) \wr Sym_{m_i},$$ where $n_i=s_im_i$  and  $Sym_{m_i}$ acts naturally on
 $\{1,\ldots,m_i\}$.
Put $$R_i=\begin{cases} M_i, \mbox{ if } n_i \mbox{ is not a power of } 2,\\
G_i, \mbox{ otherwise.}
\end{cases}$$
If $n_i$ is a power of $2$, then put $s_i=n_i$ and $m_i=1$. Thus, $$H \le_2 R = \prod\limits_{i=1}^{t} R_i \cong
\prod\limits_{i=1}^{t} Sp_{2s_i}(q_i) \wr Sym_{m_i}=N\rtimes C,$$ where $N=\prod\limits_{i=1}^{t} (Sp_{2s_i}(q_i))^{m_i}$ and
$C=\prod\limits_{i=1}^{t}Sym_{m_i}$.

Note that $N \in \mathbb{Y}_2$ by Lemma~\ref{ProdPSp2w}. Moreover, $R/N\cong C \in \mathbb{X}_2$ (see, for example,
\cite[Lemma~4]{CarterFong}).

\smallskip

Let $T \in Syl_2(N)$. Then by Lemmas~\ref{N_G(S)=S} and~\ref{NormInWrProdPSpSn},
$$N_R(T)/T \cong \prod\limits_{i=1}^{t} \mathbb{Z}_3\wr Sym_{m_i}. 
$$
Note that for each $i \in \{1,\dots,t\}$,
$3$ does not divide $m$ if $m\preceq m_i$. Indeed, since each ${q_i\equiv \pm 3\pmod 8}$, $(*)$ implies that $m_i=2^{w_i}$ or $m_i=2^{w_i}(2^{2k_i}+1)$ for some non-negative integers $w_i$ and $k_i$. Therefore,  if $m\preceq m_i$, then $m$ is ether a power of $2$ itself or $$m=2^{w_i}(2^{2k_i}+1)\equiv (-1)^{w_i+1}\pmod 3.$$ Thus, by
 Lemma~\ref{ProdAbWRSn}, $N_R(T)/T \in \mathbb{Y}_2$. Hence, $R \in \mathbb{Y}_2$ by Proposition~\ref{GeneralCriterium} and
 $H$ is pronormal in $R$.
Moreover, $N_G(S)=N_R(S)$ by Lemma~\ref{N_G(S)=S}. Thus, $H$ is pronormal in $G$ by Lemma~\ref{Overgroup}, a contradiction.

\medskip

$(1)\Rightarrow (2)$ It is easy to see that if for some $i$, $G_i$  contains a non-pronormal subgroups $H_i$ of odd index, then $G$
contains a non-pronormal subgroup $H_i \times \prod_{j\not = i} G_j$ of odd index. Thus, $G_i \in \mathbb{Y}_2$ for each $i\in \{1,\ldots,t\}$.

\medskip

$(2) \Rightarrow (3)$ follows from Lemma~\ref{PSpPRN}. \hfill $\Box$

\medskip

\noindent{\bf Remark 1.} Theorem~\ref{ProdPSp2w2w(22k+1)} provides a criterion when all the subgroups of odd index are pronormal in the direct product of symplectic groups over fields of odd characteristics. 
However, it easy follows form the proof of Theorem~\ref{ProdPSp2w2w(22k+1)} that solving General Problem for a subgroup $H$ of odd index in the direct product $G$ of symplectic groups over fields of odd characteristics is reducible to solving General Problem for the subgroup $H$ in some subgroup $$R=\prod\limits_{i=1}^{t} Sp_{2s_i}(q_i) \wr Sym_{m_i},$$ where each $s_i$ is a power of $2$, each $q_i \equiv \pm 3 \pmod{8}$, and each $Sym_{m_i}$ acts naturally on $\{1,\dots, m_i\}$, moreover, we can assume that this action is primitive. 
Proposition~\ref{GeneralCriterium} allows to reduce solving General Problem for $H$ in $R$ to solving General Problem for a subgroups $H^*$ (which depends on $H$) of odd index in the factor-group $N_R(T)/T$, where $T$ is a Sylow $2$-subgroup of $H\cap A$ and $A$ is the base of the wreath product $R$, and $$N_R(T)/T \cong \prod\limits_{i=1}^{t} \mathbb{Z}_3\wr Sym_{m_i}.$$
Moreover, $H^*$ projects naturally to each $Sym_{m_i}$, and we can assume that the image of $H^*$ is a primitive subgroup of $Sym_{m_i}$ for each $i$. Since $H^*$ is a subgroup of odd index in $N_R(T)/T$, its image in $Sym_{m_i}$ contains a transposition and, therefore, by the well-known Jordan theorem \cite[Theorem~13.3]{Wielandt}, the image of $H^*$ coincides with $Sym_{m_i}$ for each $i \in \{1,\ldots,t\}$. In Section~\ref{SPT2} we obtain a criterion which allows to answer the question, is $H^*$ pronormal in~$N_R(T)/T$?

\section{Pronormality of a subgroup of odd index in the direct product of groups of the type $\mathbb{Z}_{p}\wr Sym_{n}$}\label{SPT2}

Let $$G =\prod\limits_{i=1}^t G_i = V \rtimes B, \mbox{ where } G_i \cong \mathbb{Z}_{p_i}\wr Sym_{n_i} \mbox{ for each  } i,$$ each
$p_i$ is an odd prime, and each $Sym_{n_i}$ acts naturally on $\{1,\dots, n_i\}$; $$V=\sum_{i=1}^t V_i, \mbox{ where }
V_i=(\mathbb{Z}_{p_i})^{n_i} \mbox{ for each } i,$$ and $$B=\prod\limits_{i=1}^t B_i, \mbox{ where  } B_i=Sym_{n_i} \mbox{ for each }
i.$$

Define the map $\overline{\phantom{x}}: G \rightarrow B$ such that $\overline{g}=b$ for $g=vb$, where $v\in V$ and $b \in B$.

\smallskip

Let $\pi_i: G \rightarrow G_i$ be the projection for each $i \in \{1, \ldots, t\}$.

\smallskip

It is easy to see that $G_i = V_i \rtimes B_i$ for each $i$. Let $\overline{\phantom{x}}^i: G_i \rightarrow B_i$ be the corresponding
natural epimorphism for each $i \in \{1, \ldots, t\}$.
\smallskip

Let $\overline{\pi_i}: B \rightarrow B_i$ be the corresponding projection for each $i \in \{1, . . . , t\}$.

\smallskip

It is easy to see that for each $i \in \{1, . . . , t\}$, the corresponding diagram

$$
\begin{CD}
G @>>\pi_i> G_i\\
@VV\overline{\phantom{x}}V @VV\overline{\phantom{x}}^iV \\
B @>>\overline{\pi_i}>    B_i \\
\end{CD}
$$
is commutative.

\medskip

\bigskip

{{ In this section, we obtain a criterion of pronormality of a subgroup $H \le_2 G$ such that $\overline{\pi_i(H)}^i$ is a primitive
subgroup of $Sym_{n_i}$ for each $i \in \{1,\ldots, t\}$ or, equivalently, $\overline{H}=B$.}} Let us explain that $\overline{H}=B$ if each $\overline{\pi_i(H)}^i$ is a primitive
subgroup of $Sym_{n_i}$. Note that for each $i \in \{1,\ldots, t\}$, ${\pi_i(H) \le_2 G_i}$, therefore, ${\overline{\pi_i(H)}^i
\le_2 B_i=Sym_{n_i}}$ and so, $\overline{\pi_i(H)}^i$ contains a transposition. Thus, by \cite[Theorem~13.3]{Wielandt},
$${\overline{\pi_i(H)}^i=B_i=Sym_{n_i}} \mbox{ for each } i \in \{1,\ldots,t\}.$$ Recall that for each $i$, $B_i=Sym_{n_i}\in \mathbb{X}_2$ by \cite[Lemma~4]{CarterFong}. Thus, ${\overline{H}=\prod\limits_{i=1}^t \overline{\pi_i(H)}^i= B}$ by Lemma~\ref{XtimesYprojections}.

\bigskip



\bigskip

First, consider the case $t=1$ and specify our notation for this situation.

Let $G={\mathbb{Z}_p \wr Sym_n} ={V \rtimes B}$, where $p$ is an odd prime and $B=Sym_n$ acts naturally on $\{1,\dots, n\}$, and let
$V=(\mathbb{Z}_p)^n$ coincides with the base of the wreath product. Define $$V^{+} =\{(x,x,\ldots,x)\mid x \in \mathbb{Z}_p\} \le V$$
and $$V^{-}=\{(x_1,x_2,\ldots,x_n)\mid x_i \in \mathbb{Z}_p \mbox{ and } \sum\limits_{i=1}^n x_i=0\} \le V.$$

Since $V$ is abelian, $\overline{G}$ acts on $V$ by the natural way.

\begin{lemma}\label{WHisTheOnkyPropNormSgr} In the introduced notation, let $G=\mathbb{Z}_p \wr Sym_n$, where $p$ is an odd prime.
Assume that $\overline{H}=Sym_n$. Then the following statements hold{\rm:}

$(1)$ all the $\overline{H}$-invariant subgroups of $V$ are $0$, $V^{+}$, $V^{-}$, and $V${\rm;}

$(2)$ $[\overline{H},V^{-}]=[\overline{H},V]=[H,V^{-}]=[H,V]=V^{-}${\rm;}

$(3)$ the only proper normal subgroup of odd index in $G$ is $B V^{-}${\rm;}

$(4)$ $B V^{-}$ does not contain a proper normal subgroup of odd index{\rm;}

$(5)$ if $V^{-} \le H$, then $H \in \{B V^{-}, G\}$ and $H \unlhd G$.
\end{lemma}

 \begin{proof} 
$(1)$ Follows from \cite[Proposition~5.3.4]{Kleidman_Liebeck}, for example.

\smallskip

$(2)$ It is easy to see that $$[V^{-},H]=[V^{-}, \overline{H}]=[V^{-},B].$$
Moreover, it is clear that $$[\overline{H},V^{-}]=[H,V^{-}] \le [\overline{H},V]=[H,V]\le V^{-}.$$ Prove that $V^{-} \le [H,V^{-}]$.
Let $i,j \in \{1,\ldots,t\}$ with $i<j$ and $\langle a \rangle=\mathbb{Z}_p$. Define  $w_{ij}(a)\in V$ as follows:
$$w_{ij}(a)=(x_1,x_2,\ldots,x_n), \mbox{ where } x_i=a, x_j=-a, \mbox{ and } x_k=0 \mbox{ for } k \not \in \{i,j\}.$$
Recall that $$V^{-}=\langle w_{ij}(a) \mid 1\le i <j \le t \rangle.$$
Moreover, $[w_{ij}(a),(i,j)]=w_{ij}(-2a)$, $\langle -2a \rangle=\mathbb{Z}_p$ (since $p$ is odd), and $\overline{H}=B$ contains all
the transpositions. Now it is clear that $V^{-}\le[V^{-},B]= [V^{-},H]$.

\smallskip

$(3), (4)$ It is easy to see that $B V^{-}$ is a proper normal subgroup of odd index in~$G$. Let $K$ be a normal subgroup of odd index
in $G$ (of $B V^{-}$, respectively). Then
by the Sylow theorems, $K$ contains each Sylow $2$-subgroup of $G$ (of $B V^{-}$, respectively). In particular, $K$ contains each Sylow
$2$-subgroup of $B$ and each transposition of~$B$. Therefore, $K$ contains $B$ (which is generated by
these transpositions). Thus, by part~$(2)$ of this lemma, $K$ contains $[V^{-},B]=V^{-}$, and so, $K \ge B V^{-}$.

\smallskip

$(5)$ Assume that $V^{-}\le H$. Show that $B\le H$. Let $s$ be a transposition from $B$. Since $\overline{H}=B$, there exists $h \in H$ such
that $h=sv$ for some $v \in V$. Note that $|s|=2$ and $|v|=p$ is odd. Using elementary calculations, it is easy to show that
$h^p=s v^{-}$ for some $v^{-} \in [V,B]$. By part~$(2)$ of this lemma, $[V,B]=V^{-}$. Since $h^p\in H$ and $V^{-} \le H$, we have $s
\in H$. Thus, $H$ contains each transposition from $B$. Therefore, $B \le H$ and $BV^{-}\le H$. Now $|G: BV^{-}|=p$ implies $H \in \{B V^{-}, G\}$ and $H \unlhd G$ by part~$(3)$ of this lemma.

\end{proof}

\begin{proposition}\label{AAAA} In the introduced notation, let $G=\mathbb{Z}_p \wr Sym_n$, where $p$ is an odd prime, and $H \le G$
such that $\overline{H} = B = Sym_n$. Then the following statements hold{\rm:}

$(1)$ if $p$ does not divide $n$, then $H$ is pronormal in $G${\rm;}

$(2)$ if $p$ divides $n$, then $H$ is pronormal in $G$ if and only if $V^{-} \le H${\rm;}

$(3)$ if $H \le K <G$, then $H$ is pronormal in $K$.

\end{proposition}

\begin{proof} Note that $H \cap V$ is an $\overline{H}$-invariant subgroup of $V$, therefore, by  Lemma~\ref{WHisTheOnkyPropNormSgr}
part~$(1)$, $H \cap V \in \{0, V^{+}, V^{-}, V\}$.

\smallskip

If $p$ does not divide $n$, then by Lemma~\ref{ComplAwrSn}, $\overline{H}=B$ is pronormal in $G=VB=V\overline{H}$. Now,
Lemma~\ref{PronHomImag} implies that $H$ is pronormal in $G=VH$. Therefore, for each $K$ such that $H\le K \le G$ we have that $H$ is
pronormal in $K$ by Lemma~\ref{Overgroup}.


\smallskip

Let $p$ divides $n$. Then $n>2$ and $V^{+}< V^{-}$. Suppose that $H$ is pronormal in~$G$. Show that $V^{-} \le H$. By
Lemma~\ref{WHisTheOnkyPropNormSgr} part~$(1)$, it is sufficient to understand that $H \cap V \not = 0$ and $H \cap V \not =
V^{+}$.

Assume that $H \cap V =0$. Then $$N_V(H)=C_V(H)=C_V(\overline{H})=V^{+}.$$ By Lemma~\ref{WHisTheOnkyPropNormSgr} part~$(2)$, we
have $[H,V] = V^{-}$.
Recall that $V^{+} < V^{-}$. So, $$N_V(H)+[H,V]=V^{-}<V.$$ Lemma~\ref{CritComplAbel} implies that $H$ is not pronormal in~$VH=G$.

Assume that $H \cap V=V^{+}$. Let $u=(u_1,\dots,u_n) \in N_V(H)$. Then $[u,h] \in H \cap V=V^{+}$ for each $h \in H$. So, if we take $h$ such that
$\overline{h}$ is a transposition $(i,j)$, then we obtain that $u_i-u_j=u_j-u_i$. The oddness of $p$ implies that $u_i=u_j$. Thus,
$N_V(H) \le V^{+}$. By Lemma~\ref{WHisTheOnkyPropNormSgr} part~$(2)$, $[H,V]= V^{-}$.  Thus, $$N_V(H)+[H,V] \le V^{+}+V^{-}=V^{-}<V.$$
Lemma~\ref{CritComplAbel} implies that $H$ is not pronormal in $G$.

\medskip

Assume that $V^{-} \le H$.
By Lemma~\ref{WHisTheOnkyPropNormSgr} part~$(5)$, $H \in \{B V^{-}, G\}$ and $H \unlhd G$. Thus, $H \,prn\, G$.

\smallskip

Let $K$ be a subgroup of $G$ such that $H\le K$ and $K<G$. Then $V \not \le K$ and all the possibilities for $H$-invariant subgroups
from $K \cap V$ are $0$, $V^{+}$, and $V^{-}$. Now recall that by Lemma~\ref{WHisTheOnkyPropNormSgr} part~$(2)$,
$[H,V^{-}]=V^{-}$, and $V^{+}=C_{V^{+}}(H)\le N_{V^{+}}(H)$ since $V^{+} \le Z(G)$. Therefore, by Lemma~\ref{CritComplAbel}, $H$ is
pronormal in $(K\cap V)H=K$.

\end{proof}

\medskip

Now consider the case $t>1$. For each $i \in \{1, \ldots, t\}$, define corresponding subgroups $V_i^{+}$ and $V_i^{-}$ of $V_i$ as above.

\smallskip

Let $K$ be a subgroup of $G$ such that $H \le K\le G$. It follows from Lemma~\ref{Quot} part~$(1)$ that if for some $i$, $\pi_i(H)$ in not pronormal in
$\pi_i(K)$, then $H$ is not pronormal in $K$.

\smallskip

Suppose $\pi_i(H)$ is pronormal in $\pi_i(K)$ for each $i \in \{1, \ldots, t\}$ and show that $H$ is pronormal in $K$.
Assume that $p_i$ divides $n_i$ for some $i$. By Proposition~\ref{AAAA}, if $\pi_i(K)=G_i$, then $V^{-}_i \le \pi_i(H)$. By
Lemma~\ref{WHisTheOnkyPropNormSgr} part~$(5)$, we have $\pi_i(H) \in \{B^{\vphantom{-}}_i V^{-}_i, G_i\}$. Note that $H \cap G_i$ is a normal subgroup
in $H$. Consiquently, $$H \cap G_i = \pi_i(H \cap G_i) \unlhd \pi_i(H) \mbox{ and } |\pi_i(H):H \cap G_i| \mbox{ is odd}.$$ Parts
$(3)$ and~$(4)$ of Lemma~\ref{WHisTheOnkyPropNormSgr} imply that $B^{\vphantom{-}}_i V^{-}_i \unlhd H \cap G_i$. Moreover, $B^{\vphantom{-}}_i V^{-}_i \unlhd G$ and
by Lemma~\ref{Quot}, $H$ is pronormal in $K$ if and only if $H/(B^{\vphantom{-}}_i V^{-}_i)$ is pronormal in $K/(B^{\vphantom{-}}_i V^{-}_i)$.
Note that $G_i/B^{\vphantom{-}}_i V^{-}_i \cong \mathbb{Z}_{p_i}=\mathbb{Z}_{p_i} \wr Sym_1$ and $p_i$ does not divide 1, of course. Thus, replacing $K$ and $H$ with the corresponding quotients by the normal subgroup of $G$ generated by all $B^{\vphantom{-}}_i V^{-}_i$ such that $p_i$ divides $n_i$ and $\pi_i(K)=G_i$, we can assume that for each
$i$, if $p_i$ divides $n_i$, then $\pi_i(K) \le B^{\vphantom{-}}_i V^{-}_i$. Prove the following assertion.

\begin{proposition}\label{PronV1B1W2B2} In the introduced notation, let $$R=\prod_{i=1}^{t_0} B_i V_i \times \prod_{i=t_0+1}^{t}
B^{\vphantom{-}}_i V^{-}_i$$ be a subgroup of $G$. Additionally assume that $p_i$ does not divide $n_i$ if $1\le i\le t_0$, and $p_i$ divides $n_i$ if $t_0+1\le i \le t$. Suppose that $H \le
K \le R$ and $\overline{H}=\prod_{i=1}^{t} B_i$. Then $H$ is pronormal in $K$.

\end{proposition}

\begin{proof} By Lemma~\ref{Overgroup}, it is sufficient to prove that $H$ is pronormal in $R$. Note that $R=(R\cap V)H$.  Thus, we use Lemma~\ref{CritComplAbel}.

Let $U$ be an $H$-invariant (equivalently, $\overline{H}$-invariant) subgroup of $R \cap V$. For each $i \in \{1, \ldots, t\}$, denote
by $\sigma_i: V \rightarrow V_i$ the corresponding projection. Furthermore, consider the restriction of $\overline{\phantom{x}}$ to $H$  and denote by $H_i$ the complete preimage
of $B_i$ under the epimorphism $H\rightarrow B$.
By Lemma~\ref{WHisTheOnkyPropNormSgr} part~$(2)$, for each $i$ we have $$V^{-}_i\ge \left[V^{-}_i,H\right]\ge
\left[V^{-}_i,H^{\vphantom{-}}_i\right]=\left[V^{-}_i,\overline{H^{\vphantom{-}}_i}\right]=\left[V^{-}_i,B^{\vphantom{-}}_i\right]=V^{-}_i.$$

Therefore, $V^{-}_i = \left[V^{-}_i, H\right]$ for each $i$.

Assume that $\sigma_j(U) \not \le V^{+}_j$ for some $j$. Then by Lemma~\ref{WHisTheOnkyPropNormSgr} part~$(1)$, we have
$\sigma_j(U)\ge V^{-}_j$.  Thus, $$V^{-}_j=\left[V^{-}_j,\overline{H^{\vphantom{-}}_j}\right]=\left[V^{-}_j,H^{\vphantom{-}}_j\right]\le \left[\sigma_j(U),H_j\right]=\left[U,H_j\right] \le \left[U,H\right]  \le U.$$
So, if $\sigma_j(U) \not \le V^{+}_j$, then $V^{-}_j \le U$.

Show that each $u \in U$ can be presented in the form $u=u^{+}+u^{-}$, where $u^{-} \in [U,H]$ and $u^+ \in C_U(H)$.

Let $u=\sum\limits_{i=1}^t u_i$, where $u_i \in V_i$.

Assume $i \le t_0$. Then $$V_i=V_i^{+}\oplus V_i^{-}.$$ So, there exists
a unique decomposition $$u_i=u_i^{+}+u_i^{-}, \text{where } u_i^{+} \in V_i^{+}\text{ and }u_i^{-} \in V_i^{-}.$$ Moreover, if $u_i^{-} \not = 0$,
then $\sigma_i(U) \not \le V_i^{+}$. Therefore $V_i^{-} \le U$ and we have $$u_i^{-} \in U\cap V_i^{-}\text{ and }u_i^{+} \in U\cap V_i^{+}.$$

Let $i>t_0$. In this case $V_i^{+} < V_i^{-}$. Define

$$u_i^{+}= \begin{cases} u_i, \mbox{ if } u_i \in V_i^{+};\\
 0 , \mbox{ if } u_i \in V_i^{-} \setminus  V_i^{+};
\end{cases}
\mbox{ and }\quad
u_i^{-} = \begin{cases} 0, \mbox{ if } u_i \in V_i^{+};\\
u_i, \mbox{ if } u_i \in V_i^{-} \setminus  V_i^{+}.
\end{cases}$$

Let $$u^{+}=\sum_{i=1}^t u_i^{+} \mbox{ and } u^{-}=\sum_{i=1}^t u_i^{-}.$$
So, $u^{+}, u^{-} \in U$ and $u=u^{+}+u^{-}$. Now $u_i^{-} \in V_i^{-}$ for each $i$, and if $u_i^{-} \not = 0$ for some $i$, then
$$\left[V^{-}_i,B^{\vphantom{-}}_i\right]=\left[V^{-}_i,\overline{H^{\vphantom{-}}_i}\right]=V_i^{-} \le U.$$ Thus, if $u_i^{-} \not = 0$, then there are $w_i \in V_i^{-}\le
U$ and ${h}_i \in {H_i}$ such that $u_i^{-}=\left[w_i,\bar{h}_i\right]=\left[w_i,{h}_i\right]$. In the case $u_i^{-} = 0$, we put ${h}_i=1$
and $w_i=0$. Let $${h}=\prod\limits_{i=1}^t {h}_i\quad \mbox{ and }\quad w=\sum_{i=1}^t w_i.$$ We have
$$u^{-}=[w,\bar{h}]=[w,h]\in [U,H].$$ 

Taking into account that $V^{+}_i \le Z(R)$ for each $i \in \{1, \ldots, t\}$, we obtain that $$u^{+} \in C_U(H).$$
So, for each $u \in U$ we have the decomposition $$u=u^{+}+u^{-} \in C_U(H) + [U,H].$$ Therefore,  $$U\le C_U(H) + [U,H]\le N_U(H) + [U,H].$$ The inclusion
$N_U(H) + [U,H] \le U$ is clear. Application of Lemma~\ref{CritComplAbel} completes the proof. Thus, $H$ is pronormal in $R$.
%
%
%
%
\end{proof}

\medskip

So, we have proved the following theorem.

\begin{theorem}\label{SOIinProdC3wrSn} Let $G =\prod\limits_{i=1}^t G_i$, where  $G_i \cong \mathbb{Z}_{p_i}\wr Sym_{n_i}$ for each
$i$, each $p_i$ is an odd prime, and each $Sym_{n_i}$ acts naturally on $\{1,\dots, n_i\}$. In the introduced notation, assume that
$H$ is a {\rm(}non-trivial{\rm)} subgroup of odd index of $G$ such that $\overline{\pi_i(H)}^i$ is a primitive subgroup of $Sym_{n_i}$
for each $i \in \{1,\ldots,t\}$. Then the following statements hold{\rm:}

$(i)$ $\overline{H}=\prod\limits_{i=1}^t Sym_{n_i}${\rm;}

$(ii)$ for any $K \le G$ such that $H \le K$, $H$ is pronormal in $K$ if and only if $\pi_i(H)$ is pronormal in $\pi_i(K)$ for each $i
\in \{1,\ldots,t\}${\rm;} and

$(iii)$ if $t=1$ {\rm(}here we put $p_1=p$ and $n_1=n$ for brevity{\rm)}, then the following statements hold{\rm:}

$\mbox{ }$$\mbox{ }$$\mbox{ }$$\mbox{ }$ $(1)$ if $p$ does not divide $n$, then $H$ is pronormal in $G${\rm;}

$\mbox{ }$$\mbox{ }$$\mbox{ }$$\mbox{ }$ $(2)$ if $p$ divides $n$, then $H$ is pronormal in $G$ if and only if $V^{-} \le H${\rm;}

$\mbox{ }$$\mbox{ }$$\mbox{ }$$\mbox{ }$ $(3)$ if $H \le K <G$, then $H$ is pronormal in $K$.

\end{theorem}

\medskip

\noindent{\bf Remark 2.} In is easy to see that for each positive integer $n$, $\mathbb{Z}_2\wr Sym_n \in \mathbb{X}_2$ if $Sym_{n}$ acts naturally on $\{1,\dots, n\}$. Thus, the requirement of oddness of each $p_i$ can be omitted in parts $(i)$ and $(ii)$ of Theorem~\ref{SOIinProdC3wrSn} by Lemma~\ref{XtimesY}.

\medskip

Consider an example of application the obtained results to solving General Problem
for a given subgroup of odd index in the finite simple symplectic group $PSp_6(3)$.

\begin{example}\label{ExPSp6(3)} Let $\overline{G}=PSp_6(3)$ and $\overline{H}\le_2 \overline{G}$. Decide, is $\overline{H}$ pronormal
in $\overline{G}${\rm?}

\end{example}

Let $G=Sp_6(3)$, $V$ be the natural module of $G$, and $H$ be the complete preimage of $\overline{H}$ in $G$.
It is easy to see that $H \le_2 G$.
Recall that by Lemma~\ref{Quot}, $H$ is pronormal in $G$ if and only if $\overline{H}$ is pronormal in $\overline{G}$.

We can assume that $H$ is a proper subgroup of $G$. Therefore, there exists a maximal subgroup $M$ of $G$ such that $H \le M$. The
index $|G:M|$ is odd.

Since $q=3$ is prime, by Lemma~\ref{CMSOIPSp}, possibilities for $M$ are the following{\rm:}

\begin{itemize}
\item[Type $(1)$] $M \cong Sp_{2}(3)\times Sp_{4}(3)$ is the stabilizer of a non-degenerate subspace of dimension $2$ of $V${\rm;}

\item[Type $(2)$] $M \cong Sp_{2}(3)\wr Sym_3$ is the stabilizer of an orthogonal decomposition $${V =\bigoplus_{i=1}^3 V_i}$$ into a
    sum of $3$ pairwise isometric non-degenerate subspaces $V_i$ of dimension~$2$.
\end{itemize}

Let $S \in Syl_2(G)$ such that $S \le H$.

If there is a subgroup $M$ of Type $(1)$ such that $H \le M$, then $H$ is reducible on $V$. Moreover, $H$ is pronormal in $M$ by
Lemma~\ref{ProdPSp2w}. By Lemma~\ref{N_G(S)=S}, we have ${N_G(S)=N_M(S)}$. Therefore, $H$ is pronormal in $G$ by
Lemma~\ref{Overgroup}.

If there is no a maximal subgroup $M$ of Type $(1)$ such that $H \le M$, then by Lemma~\ref{CMSOIPSp}, $H$ is irreducible on $V$ and
there is a maximal subgroup $M$ of Type $(2)$ such that $H \le M\cong Sp_{2}(3)\wr Sym_3$, where $Sym_3$ acts naturally on the set
$\{1,2,3\}$.
By Lemma~\ref{N_G(S)=S}, we have ${N_G(S)=N_M(S)}$. Therefore, by Lemma~\ref{Overgroup}, $H$ is pronormal in $G$ if and only if $H$ is
pronormal in $M$.

By \cite{Atlas}, $Sp_{2}(3)\cong \mathbb{Z}_2.Alt_4$, therefore, $$H \le M \cong (\mathbb{Z}_2.Alt_4)\wr Sym_3.$$
Moreover, since $H$ is irreducible on $V$ and the dimension of $V_i$ was chosen as maximal as possible (since there was the only
choice $dim(V_i)=2$), we conclude that $H$ projects onto $Sym_3$.

By Lemma~\ref{Quot}, $H$ is pronormal in $M$ if and only if $H/O_2(M)$ is pronormal in $M/O_2(M)\cong \mathbb{Z}_3\wr Sym_3$.
Now to decide is $H$ pronormal in $M$, it is sufficient to apply Theorem~\ref{SOIinProdC3wrSn} (really, in this case it is sufficient to apply Proposition~\ref{AAAA} which is a part of Theorem~\ref{SOIinProdC3wrSn}). \hfill$\Box$

\medskip

Thus, Problems~1 and~2 formulated in Section~\ref{Intr} are of interest. Moreover, the following problems naturally arise.

\medskip

\noindent{\bf Problem 3.} Find a criterion of pronormality of a given subgroup of odd index in the group $G =\prod\limits_{i=1}^t
\mathbb{Z}_{p_i}\wr Sym_{n_i}$, where each $p_i$ is a prime and each $Sym_{n_i}$ acts naturally on $\{1,\dots, n_i\}$.

\bigskip

\noindent{\bf Problem 4.} Provide an effective algorithm which solves General Problem for an arbitrary subgroup of odd index in the
group $G =\prod\limits_{i=1}^t \mathbb{Z}_{p_i}\wr Sym_{n_i}$, where each $p_i$ is a prime and each $Sym_{n_i}$ acts naturally on
$\{1,\dots, n_i\}$.

\medskip

\section{Acknowledgements}

This work was supported by the Russian Science Foundation (project 19-71-10067).

\medskip




\bigskip

\noindent Natalia~V. Maslova\\
Krasovskii Institute of Mathematics and Mechanics UB RAS, Yekaterinburg,~Russia\\
Ural Federal University, Russia\\
E-mail address: butterson@mail.ru\\
ORCID: 0000-0001-6574-5335

\medskip

\noindent Danila~O. Revin\\
Krasovskii Institute of Mathematics and Mechanics UB RAS, Yekaterinburg,~Russia\\
Sobolev Institute of Mathematics SB RAS, Novosibirsk, Russia\\
Novosibirsk State University, Russia\\
E-mail address: revin@math.nsc.ru


\begin{thebibliography}{100}

\bibitem{Babai} L.~Babai, Isomorphism Problem for a Class of Point-Symmetric Structures, \emph{Acta Math.
Acad. Sci. Hungar.} \textbf{29} (1977), 329--336.

\bibitem{BMPPR}  B. Brewster, A. Mart\'{\i}nez-Pastor, M.~D. P\'{e}rez-Ramos, Pronormal subgroups of a direct product of groups,
    \emph{ J. Algebra}. \textbf{321}:6 (2009), 1734--1745.

\bibitem{CarterFong}
R.~Carter , P.~Fong , The Sylow 2-subgroups of the finite classical groups, \emph{J. Algebra.}
\textbf{1}:2 (1964), 139--151.

\bibitem{Atlas} J.~H.~Conway, et al., Atlas of Finite Groups (Clarendon Press, Oxford, 1985).


\bibitem{Dobson}  E. Dobson, On the Cayley isomorphism problem,
    \emph{Discrete Mathematics}. \textbf{247} (2002), 107--116.


\bibitem{DH}{K. Doerk, T. Hawks} \emph{Finite soluble groups} (Walter de Gruyter, Berlin, New York, 1992).


\bibitem{Guo_Mas_Rev}
W. Guo, N.~V. Maslova, D.~O. Revin, On the pronormality of subgroups of odd index in some extensions of finite groups, \emph{Siberian
Math. J.} \textbf{59}:4 (2018), 610--622.

\bibitem{GR_Surv} W.~Guo, D.~O.~Revin, Pronormality and submaximal $\mathfrak{X}$-subgroups in finite groups, \emph{Communications in
    Mathematics and Statistics}, \textbf{6}:3 (2018), 289--317.

\bibitem{Hall}
{P.~Hall}, {\it Phillip Hall lecture notes on group theory - Part 6} (University of Cambridge, Cambridge, 1951-1967. Available at:
http://omeka.wustl.edu/omeka/items/show/10788 ).

\bibitem{Kleidman_Liebeck}
P.~B. Kleidman, M. Liebeck, \emph{The subgroup structure of the finite classical groups} (CUP, Cambridge 1990).

\bibitem{Kondrat'ev}
A.~S. Kondrat'ev, Normalizers of the Sylow $2$-subgroups in finite simple groups, \emph{Math. Notes.} \textbf{78}:3 (2005), 338--346.

\bibitem{Kond_Mas_Rev1}
A.~S. Kondrat'ev, N.~V. Maslova, D.~O. Revin, On the pronormality of subgroups of odd index in finite simple groups, \emph{Siberian
Math. J.} \textbf{56}:6 (2015), 1001--1007.


\bibitem{Kond_Mas_Rev2}
A.~S. Kondrat'ev, N.~V. Maslova, D.~O. Revin, A pronormality criterion for supplements to abelian normal subgroups, \emph{Proc.
Steklov Inst. Math.} \textbf{296}:Suppl.~1
(2017), 1145--1150.


\bibitem{Kond_Mas_Rev3}
A.~S. Kondrat'ev, N.~V. Maslova, D.~O. Revin, On the pronormality of subgroups of odd index in finite simple symplectic groups,
\emph{Siberian Math. J.} \textbf{58}:3 (2017), 467--475.

\bibitem{Kond_Mas_Rev4}
A.~S. Kondrat'ev, N.~V. Maslova, D.~O. Revin, On pronormal subgroups in finite simple groups, \emph{Doklady Mathematics},
\textbf{482}:1 (2018), 405--408.


\bibitem{KMR_Surv}
A.~S.~Kondrat'ev, N.~V.~Maslova, D.~O.~Revin, On the pronormality of subgroups of odd index in finite simple groups // Groups St
Andrews 2017 in Birmingham,  (Birmingham, 5th-13th August 2017), London Mathematical Society Lecture Note Series, 455, eds. C. M.
Campbell, M. R. Quick, C. W. Parker, E. F. Robertson, C. M. Roney-Dougal, Cambridge University Press, Cambridge, 2019, 406--418.

\bibitem{Kond_Mas_Rev5}
A.~S. Kondrat'ev, N.~V. Maslova, D.~O. Revin, Finite simple exceptional groups of Lie type in which all the
subgroups of odd index are pronormal, \emph{J. Group Theory}, to appear; arXiv:1910.02524 [math.GR].

\bibitem{Maslova1}
N.~V. Maslova, Classification of maximal subgroups of odd index in finite simple classical groups, \emph{Proc. Steklov Inst. Math.}
\textbf{267}:Suppl.~1 (2009), 164--183.

\bibitem{Maslova10}
N.~V. Maslova, Classification of maximal subgroups of odd index in finite simple classical groups: Addendum: \emph{Siberian Electronic
Mathematical Reports}, \textbf{15} (2018), 707--718.






\bibitem{Palfy}
P.~P.~Palfy, Isomorphism Problem for Relational Structures with a Cyclic Automorphism,
\emph{Europ. J. Combinatorics.} \textbf{8} (1987), 35--43.

\bibitem{Peng}
T.~A.~Peng, Pronormality in finite groups, \emph{J. London Math. Soc.} \textbf{8} (1971), 301--306.

\bibitem{Praeger} 
Ch.~E.~Praeger, On Transitive Permutation Groups With a Subgroup Satisfying a Certain
Conjugacy Condition, \emph{J. Austral. Math. Soc.} \textbf{36}:1 (1984), 69--86.



\bibitem{Vdovin_Revin}
E.~P. Vdovin, D.~O. Revin, Pronormality of Hall subgroups in finite simple groups, \emph{Sib. Math. J.} \textbf{53}:3 (2012),
419--430.

\bibitem{Vdovin_Revin_2015}
E.~P. Vdovin, D.~O. Revin, The existence of pronormal $\pi$-Hall subgroups in $E_\pi$-groups, \emph{Sib. Math. J.} \textbf{56}:3 (2015),
379--383.

\bibitem{Wielandt} {H. Wielandt} \emph{Finite permutation groups} (Academic Press, London, 1964).


\end{thebibliography}
\end{document}